\documentclass[11pt]{article}

\usepackage[T1]{fontenc}
\usepackage{mathtools, amssymb, amsthm}
\usepackage[margin=2cm]{geometry}
\usepackage[hidelinks]{hyperref}

\newtheorem{theorem}{Theorem}[section]
\newtheorem{lemma}[theorem]{Lemma}

\newtheorem{corollary}[theorem]{Corollary}
\theoremstyle{definition}

\theoremstyle{remark}
\newtheorem{remark}[theorem]{Remark}
\numberwithin{equation}{section}

\newcommand{\R}{\mathbb{R}}
\newcommand{\C}{\mathbb{C}}
\newcommand{\Q}{\mathbb{Q}}
\newcommand{\Z}{\mathbb{Z}}
\newcommand{\Leb}{\mathcal L}
\newcommand{\Haus}{\mathcal H}
\newcommand{\eps}{\varepsilon}

\title{Non-rotationally invariant generalised geodesics in the disc}
\author{Maja Gw\'{o}\'{z}d\'{z}\\
ETH Z\"{u}rich\\
\texttt{mgwozdz@ethz.ch}}
\date{}

\begin{document}

\maketitle

\begin{abstract}
We study Brenier's relaxed least-action problem in the unit disc
\(D:=\{x\in\R^2:\ |x|<1\}\) at the critical time \(T=\pi\), with the endpoint
pair \(i_D\) (the identity on \(D\)) and \(-i_D\). We prove that the critical
energy shell
\[
        S^3=\{(x,v)\in\R^2\times\R^2:
        |x|^2+|v|^2=1\}
\]
supports stationary action-minimising generalised incompressible flows that
are not invariant under physical rotations. This answers the question posed by
Bernot, Figalli, and Santambrogio. Our construction relies on the normalised
surface measure \(\sigma\) on \(S^3\) and the Hopf quotient
\[
        \Pi=(N,M,L):S^3\to\mathbb S^2_{1/2},
        \qquad
        \mathbb S^2_{1/2}:=\{(n,m,\ell)\in\R^3:
        n^2+m^2+\ell^2=1/4\}.
\]
Here,
\[
        N=\frac12(x_1^2+v_1^2-x_2^2-v_2^2),\quad
        M=x_1x_2+v_1v_2,\quad
        L=x_1v_2-x_2v_1 .
\]
This quotient is a first integral of the harmonic-oscillator flow. Let \(\tau:=\Pi_\#\sigma\), let \(g\) be a bounded Borel function on \(\mathbb S^2_{1/2}\) whose \(L^\infty(\tau)\)-class is odd under \((n,m,\ell)\mapsto(n,m,-\ell)\), and let \(\delta\in\R\). If
\(1+\delta g\circ\Pi\ge0\) \(\sigma\)-a.e., then
\[
        d\mu_{\delta,g}=\pi(1+\delta g\circ\Pi)\,d\sigma
\]
has Lebesgue spatial marginal and is stationary. The induced path measure is
then minimising. We also determine exactly which members of this tilted
family are rotationally invariant. If \(\delta\ne0\), this is equivalent to
axisymmetry of the \(L^\infty(\tau)\)-class of \(g\), modulo
\(\tau\)-null sets. In particular, taking \(g(n,m,\ell)=\ell m\) with
\(0<|\delta|<8\) gives strictly positive non-rotationally invariant
minimisers.
\end{abstract}



\section{Introduction}

Let us set
\[
        D:=\{x\in\R^2:\ |x|<1\}.
\]
We rely on Arnold's geodesic interpretation of the incompressible Euler equations
and Brenier's relaxed least-action formulation \cite{Arnold1966,Brenier1989,Brenier1993,Brenier1999}. In the disc \(D\), the
problem is the relaxed endpoint problem from \(i_D\) to \(-i_D\) on
the time interval \([0,\pi]\). In this paper, the area measure is
\(\Leb^2\!\lfloor D\), so \(\Leb^2(D)=\pi\). This implies that admissible incompressible
path measures have mass \(\pi\), and we use the same convention for
incompressible phase measures. If we divide all measures by \(\pi\), we recover the normalised-area convention. We also write \(\Haus^k\) for \(k\)-dimensional
Hausdorff measure. For a path measure, the action is
\[
        \mathcal A(\eta)
        :=
        \int_{\Omega}
        \mathcal E(\omega)\,d\eta(\omega),
        \qquad
        \Omega:=C([0,\pi],\overline D),
\]
where
\[
        \mathcal E(\omega)
        :=
        \begin{cases}
        \displaystyle
        \int_0^\pi \frac12|\dot\omega(t)|^2\,dt,
        & \omega\in H^1([0,\pi],\R^2),\\[1.2em]
        +\infty,
        & \text{otherwise.}
        \end{cases}
\]
We minimise over finite positive Borel measures \(\eta\) on \(\Omega\) such
that
\[
        (e_t)_\#\eta=\Leb^2\!\lfloor D
        \qquad\forall t\in[0,\pi],
\]
and
\[
        (e_0,e_\pi)_\#\eta=(x,-x)_\#(\Leb^2\!\lfloor D),
\]
where \(e_t(\omega)=\omega(t)\). For related geodesics in spaces of
measure-preserving plans and maps, see \cite{AF09}.

For the endpoint \(i_D\mapsto -i_D\), there are two rigid rotations, that is,
\(G_\pm(t,x)=R_{\pm t}x\). Here, we use \(R_\theta\) to denote the anticlockwise rotation through angle \(\theta\), with the convention that
\[
        R_\theta(x_1,x_2)
        =
        (x_1\cos\theta-x_2\sin\theta,\,
         x_1\sin\theta+x_2\cos\theta).
\]
The maps \(G_\pm\) are calibrated by \(p(x)=|x|^2/2\). More precisely, whenever
\(\gamma\) is absolutely continuous with \(\gamma(0)=a\) and \(\gamma(\pi)=-a\), then
\[
        \int_0^\pi
        \left(\frac12|\dot\gamma(t)|^2-\frac12|\gamma(t)|^2\right)\,dt
        \ge0,
\]
where equality holds exactly for curves of the form \(a\cos t+b\sin t\). This
basic calibration is an optimality result we will need below. We write
\[
        \Phi_t(x,v):=x\cos t+v\sin t,
        \qquad
        \Phi(x,v)(t):=\Phi_t(x,v),
\]
and
\[
        \phi_t(x,v):=
        (x\cos t+v\sin t,\,-x\sin t+v\cos t).
\]
We also write
\[
        \pi_x:\R^2_x\times\R^2_v\to\R^2_x,
        \qquad \pi_x(x,v)=x,
\]
for the projection onto space. The previous parametrisation is the harmonic-oscillator example of Brenier's invariant phase-measure criterion \cite[Proposition 6.1]{Brenier1989}. It is also the parametrisation used in \cite[Lemma 2.3 and Section 4]{BFS}, up to the normalisation of Lebesgue measure and the open/closed-disc convention. For the pressure \(p(x)=|x|^2/2\), the calibrated curves are completely described by their initial positions and velocities. We shall use the following result in the
mass-\(\pi\) convention. Let \(\mu\) be a finite phase measure on the critical
shell with mass \(\pi\). Suppose that the oscillator paths remain in
\(\overline D\) for \(\mu\)-almost every initial condition and that
\[
        (\pi_x\circ\phi_t)_\#\mu=\Leb^2\!\lfloor D
        \qquad\forall t\in[0,\pi].
\]
It follows that \(\Phi_\#\mu\) is an admissible minimiser, which we prove below.
We thus reduce the problem to the construction of stationary phase measures on the
critical shell which keep the spatial marginal fixed but do not preserve physical rotations. The available freedom lies in non-axisymmetric first-integral densities on the
critical shell. The spatial marginal remains due to the fact that the oddness of the Hopf
variable \(L\) cancels on the velocity fibres. In the normalised convention of
\cite[Problem 1.2 and (1.73)]{DaneriFigalli2013}, the spatial Lebesgue measure
has total mass one, and the respective probability shell law is \(\sigma\). In our convention, the isotropic shell law is \(\mu_0=\pi\sigma\).

On the critical shell, the oscillator flow admits the Hopf quotient
\[
        \Pi=(N,M,L):S^3\to\mathbb S^2_{1/2}.
\]
Observe that this quotient is constant along oscillator orbits. It is the standard Hopf
fibration in oscillator variables, with target sphere of radius \(1/2\). Let us
compare it with
\[
        \pi_P(x,v)=\left(|x|^2-\frac12,\,x\cdot v\right)
\]
from \cite{BFS}. The map \(\pi_P\) is adapted to
rotationally invariant disintegrations, whereas \(\Pi\) is adapted to
stationary disintegrations. We deduce that the two quotients describe different symmetries of the same critical shell. In the case when the tilt is odd in \(L\), the spatial
marginal is preserved by fibrewise cancellation under \(v\mapsto -v\).
This allows for non-axisymmetric tilts on the Hopf sphere, which break physical rotational invariance while keeping incompressibility intact.

We say that a phase measure \(\mu\) is incompressible if
\[
        (\pi_x\circ\phi_t)_\#\mu=\Leb^2\!\lfloor D
        \qquad\forall t\in[0,\pi]
\]
holds. Following \cite[Definitions 4.1--4.2]{BFS}, we say that \(\mu\) is
stationary if \((\phi_t)_\#\mu=\mu\), and rotationally invariant if
\[
        (\overline R_\theta)_\#\mu=\mu
        \qquad\forall\theta,
        \qquad
        \overline R_\theta(x,v):=(R_\theta x,R_\theta v),
\]
where \(R_\theta\) is the anticlockwise rotation of \(\R^2\) through
angle \(\theta\). Bernot, Figalli, and Santambrogio asked whether the endpoint problem
\(i_D\mapsto -i_D\) has a non-rotationally invariant geodesic
\cite[p. 144, before Section 4.1]{BFS}. Daneri--Figalli reiterate the same question in terms of non-rotationally invariant minimal measures \cite[p. 37, end of Section 1.4.4]{DaneriFigalli2013}. We show that such examples occur already in the stationary single-shell class
\[
        S^3=\{(x,v): |x|^2+|v|^2=1\}.
\]
In our main result, Theorem \ref{thm:first-integral-tilts}, we establish a first-integral tilting criterion and, within the same family, the exact criterion for rotational invariance. The choice \(g(n,m,\ell)=\ell m\) gives the density \(1+\varepsilon LM\), which is strictly positive for \(0<|\varepsilon|<8\). This yields the polynomial example in
Corollary \ref{cor:polynomial-example}.

We note that the uniqueness results with fixed orientation of \cite[Corollary 4.6 and Theorem 4.20]{BFS} use different hypotheses. Our examples are stationary and supported on \(S^3\), but every strictly positive member of our family charges both orientations. We compare the normalisations and orientation classes in Section \ref{sec:BFS-comparison}.

In Section 2, we analyse the Hopf quotient and the rotation action induced on
it, and in Section 3, we prove the tilting criterion. Section 4 derives the
infinite-dimensional family and the polynomial example. Finally, in Section 5, we
compare the construction with the orientation-fixed uniqueness results of
Bernot--Figalli--Santambrogio \cite{BFS}.

\section{The Hopf quotient of the oscillator flow}

We begin with the analysis of the geometry of the critical shell. In the first part, we rely on surface disintegration, which also gives the spatial marginal. We then apply the Hopf quotient and the observation that it is constant along oscillator orbits and, additionally, keeps track of the remaining variables on the shell. Finally, we compute the action induced by physical rotations on this quotient. To this end, we set
\[
        S^3:=\{(x,v)\in\R^2\times\R^2:\ |x|^2+|v|^2=1\},
\]
and let
\[
        \sigma:=\frac{1}{\Haus^3(S^3)}\Haus^3\!\lfloor S^3
\]
be the normalised surface measure. This implies \(\sigma(S^3)=1\), and
\(\Haus^3(S^3)=2\pi^2\). Notice that the set \(S^3\cap(\partial D\times\R^2)\) is \(\sigma\)-null. Therefore, every absolutely continuous shell measure below may be analysed either on \(S^3\) or on \(S^3\cap(D\times\R^2)\), and this convention does not affect push-forwards or the action.

For later use, we write
\[
        x\wedge v:=x_1v_2-x_2v_1 .
\]
Indeed, suppose that \((x,v)\in S^3\) and that
\(t\mapsto x\cos t+v\sin t\) reaches \(\partial D\). It follows that the Gram matrix of
\((x,v)\) has eigenvalue \(1\). Its trace is \(1\), so its other eigenvalue is \(0\). We then obtain \(x\wedge v=0\). Hence the boundary-hitting oscillators are contained in
\(\{x\wedge v=0\}\). Moreover, the function \((x,v)\mapsto x\wedge v\) is non-zero and
real-analytic on \(S^3\). Its zero set has zero \(\Haus^3\)-measure. This implies that
\(\{x\wedge v=0\}\), including the lower-dimensional fibres \(x=0\) and
\(|x|=1\), is \(\sigma\)-null.

\begin{lemma}[Surface disintegration and spatial marginal]
\label{lem:surface-disintegration}
On the full-measure, we set \(\{0<|x|<1\}\), and use the coordinates
\[
        x=r(\cos\theta,\sin\theta),
        \qquad
        v=\sqrt{1-r^2}(\cos\alpha,\sin\alpha),
\]
where \(r\in(0,1)\) and \(\theta,\alpha\in[0,2\pi)\). We obtain
\[
        d\sigma
        =
        \frac{1}{2\pi^2}\,r\,dr\,d\theta\,d\alpha.
\]
It follows that
\[
        (\pi_x)_\#\sigma=\frac1\pi \Leb^2\!\lfloor D,
\]
and, with
\[
        C_x:=\{v\in\R^2:\ |v|=\sqrt{1-|x|^2}\},
\]
\[
        \pi\sigma(dx,dv)
        =
        \frac{1}{2\pi\sqrt{1-|x|^2}}\,
        (\Haus^1\!\lfloor C_x)(dv)
        \,\Leb^2\!\lfloor D(dx).
\]
\end{lemma}

\begin{proof}
In these coordinates, the induced metric has Gram matrix
\[
        \operatorname{diag}\bigl((1-r^2)^{-1},r^2,1-r^2\bigr),
\]
and determinant \(r^2\). It follows that
\[
        d\Haus^3\!\lfloor S^3=r\,dr\,d\theta\,d\alpha.
\]
Since \(\Haus^3(S^3)=2\pi^2\), this gives the formula for \(d\sigma\). If we now integrate in \(\alpha\), we obtain the spatial marginal. Because \(d\Haus^1\!\lfloor C_x=\sqrt{1-|x|^2}\,d\alpha\) holds, the disintegration of \(\pi\sigma\) follows.
\end{proof}

In particular, it follows that
\[
        \mu_0:=\pi\sigma
\]
has spatial marginal \(\Leb^2\!\lfloor D\). In probability normalisation, Brenier's shell law \cite[(6.15)]{Brenier1989} is
\[
        \pi^{-2}\delta(|x|^2+|v|^2-1)\,dx\,dv .
\]
From
\[
        \delta(|y|^2-1)\,dy
        =
        \frac12\,\Haus^3\!\lfloor S^3(dy)
        \qquad\text{in }\R^4,
\]
we infer that this probability shell law is exactly \(\sigma\). Multiplication by \(\pi\) gives the mass-\(\pi\) convention we use here. We now pass from the shell to paths. To this end, let us define the following continuous map
\[
        \Phi:S^3\to C([0,\pi],\overline D),
        \qquad
        \Phi(x,v)(t):=x\cos t+v\sin t.
\]
For \((x,v)\in S^3\), by Cauchy--Schwarz, we obtain
\[
        |x\cos t+v\sin t|
        \le (|x|^2+|v|^2)^{1/2}
             (\cos^2t+\sin^2t)^{1/2}
        =1,
\]
so \(\Phi\) takes values in \(C([0,\pi],\overline D)\). We now introduce the quadratic quantities
\[
        N(x,v):=\frac12(x_1^2+v_1^2-x_2^2-v_2^2),
\]
\[
        M(x,v):=x_1x_2+v_1v_2,
        \qquad
        L(x,v):=x_1v_2-x_2v_1,
\]
and
\[
        \Pi:S^3\to\R^3,
        \qquad
        \Pi(x,v):=(N(x,v),M(x,v),L(x,v)).
\]
We also set
\[
        \mathbb S^2_{1/2}:=
        \{(n,m,\ell)\in\R^3:\ n^2+m^2+\ell^2=1/4\}.
\]

\begin{lemma}[Hopf quotient]\label{lem:hopf-quotient}
The map \(\Pi\) sends \(S^3\) onto \(\mathbb S^2_{1/2}\), and
\[
        N^2+M^2+L^2=\frac14
        \qquad\text{on }S^3.
\]
It is invariant under the oscillator flow:
\[
        \Pi\circ\phi_t=\Pi
        \qquad\forall t\in\R,
\]
and the fibres of \(\Pi\) are exactly the oscillator orbits.
\end{lemma}

\begin{proof}
Let us set
\[
        z_1:=x_1+i v_1,
        \qquad
        z_2:=x_2+i v_2.
\]
This implies that \(|z_1|^2+|z_2|^2=1\) on \(S^3\), while
\[
        N=\frac12(|z_1|^2-|z_2|^2),
        \qquad
        z_1\overline{z_2}=M-iL.
\]
It follows that
\[
\begin{aligned}
        N^2+M^2+L^2
        &=
        \frac14(|z_1|^2-|z_2|^2)^2+|z_1|^2|z_2|^2  \\
        &=
        \frac14(|z_1|^2+|z_2|^2)^2
        =
        \frac14.
\end{aligned}
\]
For the other direction, let \((n,m,\ell)\in\mathbb S^2_{1/2}\). If \(1/2+n>0\), we set
\[
        z_1=\sqrt{\frac12+n},
        \qquad
        z_2=\frac{m+i\ell}{z_1}.
\]
We obtain \(z_1\overline{z_2}=m-i\ell\), and \(|z_2|^2=(m^2+\ell^2)/(1/2+n)=1/2-n\). If \(n=\pm1/2\), then \(m=\ell=0\). The two endpoint cases follow by taking
\((z_1,z_2)=(1,0)\) and \((z_1,z_2)=(0,1)\), respectively. If we now write \(z_j=x_j+i v_j\), we obtain a point of \(S^3\) with Hopf coordinates \((n,m,\ell)\). Observe that under \(\phi_t\), both complex coordinates acquire the same phase, that is,
\[
        z_j\mapsto e^{-it}z_j.
\]
It follows that \(|z_1|^2\), \(|z_2|^2\), and \(z_1\overline{z_2}\) are
invariant. Hence \(\Pi\circ\phi_t=\Pi\).

It remains to identify the fibres. Let \(z,w\in S^3\subset\C^2\) have the same
Hopf coordinates. Then
\[
        |z_1|^2=|w_1|^2,\qquad |z_2|^2=|w_2|^2,
        \qquad z_1\overline{z_2}=w_1\overline{w_2}.
\]
In other terms, \(zz^*=ww^*\) as rank-one Hermitian matrices. Since
\(|z|=|w|=1\), we obtain \(w=e^{i\tau}z\) for some \(\tau\in\R\). We conclude that \(z\)
and \(w\) lie on the same oscillator orbit.
\end{proof}

In the next lemma, we analyse the push-forward of \(\sigma\) by the Hopf quotient.

\begin{lemma}[Hopf push-forward of surface measure]\label{lem:hopf-pushforward}
Let
\[
        \tau:=\Pi_\#\sigma.
\]
It follows that
\[
        \tau
        =
        \frac{\Haus^2\!\lfloor \mathbb S^2_{1/2}}
        {\Haus^2(\mathbb S^2_{1/2})}.
\]
If \(P(n,m,\ell):=(n,m)\), then
\[
        P_\#\tau
        =
        \frac{\mathbf 1_{\{n^2+m^2<1/4\}}}
        {\pi\sqrt{1/4-n^2-m^2}}\,dn\,dm.
\]
In particular, \(P_\#\tau\) and \(\Leb^2\!\lfloor \overline B_{1/2}(0)\) have the same null sets.
\end{lemma}

\begin{proof}
Let us use the coordinates
\[
\begin{gathered}
        z_1=e^{ia}\cos u,\qquad z_2=e^{ib}\sin u,\\
        0<u<\frac{\pi}{2},\qquad a,b\in[0,2\pi).
\end{gathered}
\]
In these coordinates,
\[
        d\sigma
        =
        \frac{1}{2\pi^2}\sin u\cos u\,du\,da\,db.
\]
We set \(\beta:=b-a\). Since
\[
        z_1\overline{z_2}
        =
        \frac12\sin(2u)e^{-i\beta},
\]
we have
\[
\begin{gathered}
        N=\frac12\cos(2u),\\
        M=\frac12\sin(2u)\cos\beta,\qquad
        L=\frac12\sin(2u)\sin\beta.
\end{gathered}
\]
We now change the variables from \((a,b)\) to \((a,\beta)\), and then integrate in
\(a\), which yields
\[
        d\tau
        =
        \frac{1}{2\pi}\sin(2u)\,du\,d\beta.
\]
Since \(dn=-\sin(2u)\,du\), this becomes
\[
        d\tau
        =
        \frac{1}{2\pi}\,|dn|\,d\beta .
\]
In the coordinates \((n,\beta)\), the surface element on \(\mathbb S^2_{1/2}\) is \(\frac12\,|dn|\,d\beta\). Since \(\Haus^2(\mathbb S^2_{1/2})=\pi\), we have obtained normalised surface measure on \(\mathbb S^2_{1/2}\).

It remains to compute the projected density. Let us first write the sphere over the
\((n,m)\)-disc as the two graphs
\[
        \ell=\pm\sqrt{1/4-n^2-m^2}.
\]
On each sheet, the surface-area factor is
\[
        \frac{1/2}{\sqrt{1/4-n^2-m^2}}\,dn\,dm.
\]
It now suffices to sum the two sheets and divide by \(\Haus^2(\mathbb S^2_{1/2})=\pi\) to obtain the desired density for \(P_\#\tau\).
\end{proof}

Finally, we compute the physical rotation action in Hopf variables.

\begin{lemma}[Physical rotations on the quotient]\label{lem:rotation-quotient}
Let \(R_\theta\) be anticlockwise rotation of \(\R^2\) through angle
\(\theta\), and let
\[
        \overline R_\theta(x,v):=(R_\theta x,R_\theta v).
\]
It follows that
\[
        \Pi\circ\overline R_\theta=\rho_\theta\circ\Pi,
\]
where
\[
        \rho_\theta(n,m,\ell)
        =
        (n\cos2\theta-m\sin2\theta,\,
        n\sin2\theta+m\cos2\theta,\,
        \ell).
\]
\end{lemma}

\begin{proof}
Let us set
\[
        a:=x_1^2+v_1^2,\qquad b:=x_2^2+v_2^2,\qquad c:=x_1x_2+v_1v_2.
\]
We obtain \(N=(a-b)/2\) and \(M=c\). We then apply the same rotation to \(x\) and \(v\),
which yields
\[
        N\mapsto N\cos2\theta-M\sin2\theta,
        \qquad
        M\mapsto N\sin2\theta+M\cos2\theta.
\]
Notice that the determinant \(L=x_1v_2-x_2v_1\) is unchanged by physical rotations. We conclude that \(L\) is fixed.
\end{proof}

\begin{remark}[Axisymmetry modulo null sets]
Let \(\tau=\Pi_\#\sigma\). We say that a bounded Borel function \(g:\mathbb S^2_{1/2}\to\R\) is axisymmetric modulo \(\tau\)-null sets if, for every \(\theta\in\R\), it satisfies
\[
        g=g\circ\rho_\theta
        \qquad \text{in }L^\infty(\mathbb S^2_{1/2},\tau).
\]
The exceptional null set may depend on \(\theta\).
\end{remark}

\section{First-integral tilts}

We now establish the tilting criterion, which requires the following calibration estimate on \([0,\pi]\).

\begin{lemma}[Pathwise calibration]
Let \(a\in\R^2\), and let \(\gamma\in H^1([0,\pi],\R^2)\) satisfy
\[
        \gamma(0)=a,
        \qquad
        \gamma(\pi)=-a.
\]
It follows that
\[
        J(\gamma):=
        \int_0^\pi
        \left(
        \frac12|\dot\gamma(t)|^2-\frac12|\gamma(t)|^2
        \right)\,dt
        \ge 0.
\]
Moreover, equality holds exactly for curves of the form
\[
        \gamma(t)=a\cos t+b\sin t,
        \qquad b\in\R^2.
\]
\end{lemma}

\begin{proof}
Let us set
\[
        \omega(t):=a\cos t,
        \qquad
        \zeta(t):=\gamma(t)-\omega(t).
\]
We obtain \(\zeta\in H^1_0([0,\pi],\R^2)\). Since \(\ddot\omega=-\omega\), integration by parts gives
\[
        \int_0^\pi
        \left(
        \dot\omega\cdot\dot\zeta-\omega\cdot\zeta
        \right)\,dt
        =
        0,
\]
and also \(J(\omega)=0\). Therefore,
\[
        J(\gamma)
        =
        \frac12\int_0^\pi
        \left(
        |\dot\zeta(t)|^2-|\zeta(t)|^2
        \right)\,dt.
\]
By sharp Wirtinger inequality on \([0,\pi]\), we obtain
\[
        \int_0^\pi|\zeta(t)|^2\,dt
        \le
        \int_0^\pi|\dot\zeta(t)|^2\,dt,
\]
and also \(J(\gamma)\ge0\). Equality holds exactly when each scalar component
is an equality case in Wirtinger's inequality. It follows that each component is a multiple of \(\sin t\), or, in other words, that \(\zeta(t)=b\sin t\) for some \(b\in\R^2\).
\end{proof}

We write
\[
        (\mathcal R_\theta\gamma)(t):=R_\theta\gamma(t)
\]
for the rotation action induced on path space. We now discuss the main construction. The theorem consists of two parts. Recall that the first provides a first-integral tilting criterion that preserves incompressibility and minimality. The second describes an exact symmetry criterion within the same tilted family. All shell measures we consider below are finite positive Borel measures, and \(\sigma\) denotes the normalised surface measure on \(S^3\).

\begin{theorem}[First-integral tilting criterion]
\label{thm:first-integral-tilts}
Let \(\tau:=\Pi_\#\sigma\), and let \(g:\mathbb S^2_{1/2}\to\R\) be bounded and Borel. We further assume that the \(L^\infty(\tau)\)-class of \(g\) is \(\kappa\)-odd, namely
\[
        g\circ\kappa=-g
        \qquad \tau\text{-a.e.},
        \qquad
        \kappa(n,m,\ell):=(n,m,-\ell).
\]
Let \(\delta\in\R\) be such that \(1+\delta g\circ\Pi\ge0\) \(\sigma\)-a.e. This condition is satisfied, for example, whenever \(g\not\equiv0\) in \(L^\infty(\tau)\) and
\[
        |\delta|\le \|g\|_{L^\infty(\tau)}^{-1}.
\]
If \(g=0\) in \(L^\infty(\tau)\), the condition is void. We define
\[
        d\mu_{\delta,g}:=\pi(1+\delta\,g\circ\Pi)\,d\sigma,
        \qquad
        \eta_{\delta,g}:=\Phi_\#\mu_{\delta,g}.
\]
We interpret \(\eta_{\delta,g}\) as a measure on \(C([0,\pi],\overline D)\). It follows that
\(\mu_{\delta,g}\) is a positive finite stationary Borel measure on \(S^3\subset\R^2_x\times\R^2_v\). Since
\(S^3\cap(\partial D\times\R^2)\) is \(\mu_{\delta,g}\)-null, we may also
view \(\mu_{\delta,g}\) as a phase measure on \(D\times\R^2\). Moreover,
\[
        \mu_{\delta,g}(S^3)=\pi,
        \qquad
        (\pi_x)_\#\mu_{\delta,g}=\Leb^2\!\lfloor D.
\]
It follows that the path measure \(\eta_{\delta,g}\) is an action-minimising
generalised incompressible flow from \(i_D\) to \(-i_D\) on \([0,\pi]\).

If \(\delta\ne0\), then \(\mu_{\delta,g}\) is rotationally invariant if and
only if
\[
        g=g\circ\rho_\theta
        \qquad
        \text{in }L^\infty(\mathbb S^2_{1/2},\Pi_\#\sigma)
        \quad\forall\theta\in\R.
\]
The same condition describes the rotational invariance of \(\eta_{\delta,g}\) under the path-space action \(\mathcal R_\theta\).
\end{theorem}

\begin{proof}
Let \(\kappa(n,m,\ell):=(n,m,-\ell)\). By Lemma \ref{lem:hopf-pushforward}, \(\tau\) is normalised surface measure on \(\mathbb S^2_{1/2}\), and \(\kappa_\#\tau=\tau\). We first remove a representative issue for \(g\). Define
\[
        \widetilde g(q):=\frac12\bigl(g(q)-g(\kappa q)\bigr).
\]
This implies that \(\widetilde g\) is bounded, Borel, and pointwise \(\kappa\)-odd.
Moreover, \(\widetilde g=g\) \(\tau\)-a.e. Indeed, let \(E\) be a \(\tau\)-null set outside which \(g\circ\kappa=-g\) holds. Since \(\kappa_\#\tau=\tau\), the set \(E\cup\kappa^{-1}(E)\) is still \(\tau\)-null. Outside this union, we have \(g(q)=-g(\kappa q)\), and thus
\(\widetilde g(q)=g(q)\). Given that \(\Pi_\#\sigma=\tau\), it follows that
\[
        (\widetilde g-g)\circ\Pi=0\qquad \sigma\text{-a.e.}
\]
We infer that \(g\) and \(\widetilde g\) define the same measure, and the
non-negativity hypothesis remains unchanged. Therefore, we may assume from now on that \(g\) is pointwise \(\kappa\)-odd. Let us set
\[
        G:=g\circ\Pi .
\]
We first compute the spatial marginal. For \(\Leb^2\)-a.e. \(x\in D\), namely, on
\(0<|x|<1\), let
\[
        C_x:=\{v\in\R^2:\ |v|=\sqrt{1-|x|^2}\}.
\]
By Lemma \ref{lem:surface-disintegration}, we have
\[
        d\sigma(x,v)
        =
        \frac{1}{2\pi^2\sqrt{1-|x|^2}}\,
        (\Haus^1\!\lfloor C_x)(dv)\,d\Leb^2(x)
\]
on \(\{0<|x|<1\}\). The involution \(v\mapsto -v\) preserves this conditional
measure and satisfies
\[
        N(x,-v)=N(x,v),\qquad M(x,-v)=M(x,v),\qquad L(x,-v)=-L(x,v).
\]
Therefore, \(G(x,-v)=-G(x,v)\) on \(S^3\). As a result, the inner fibre integral
vanishes: for every Borel set \(A\subset D\),
\[
\begin{aligned}
        \int_{\pi_x^{-1}(A)}G\,d\sigma
        &=
        \int_A
        \frac{1}{2\pi^2\sqrt{1-|x|^2}}
        \int_{C_x}G(x,v)\,d\Haus^1(v)
        \,d\Leb^2(x)                              \\
        &=0.
\end{aligned}
\]
It follows that \((\pi_x)_\#(G\sigma)=0\) as a signed Borel measure on \(D\). Since
\((\pi_x)_\#\sigma=\pi^{-1}\Leb^2\!\lfloor D\), we obtain
\[
        (\pi_x)_\#\mu_{\delta,g}=\Leb^2\!\lfloor D.
\]
We now take \(A=D\) and use the fact that the exceptional fibres over \(x=0\) and
\(|x|=1\) are \(\sigma\)-negligible, which yields \(\mu_{\delta,g}(S^3)=\pi\). Positivity is exactly our assumed non-negativity of \(1+\delta g\circ\Pi\).

We next verify stationarity and admissibility. By Lemma \ref{lem:hopf-quotient}, we have \(G\circ\phi_t=G\). Moreover, \(\phi_t\) is the restriction to \(S^3\)
of an orthogonal linear map of \(\R^4\). Therefore, \((\phi_t)_\#\sigma=\sigma\),
and
\[
        (\phi_t)_\#\mu_{\delta,g}=\mu_{\delta,g}.
\]
As a result,
\[
        (e_t)_\#\eta_{\delta,g}
        =
        (\pi_x\circ\phi_t)_\#\mu_{\delta,g}
        =
        (\pi_x)_\#\mu_{\delta,g}
        =
        \Leb^2\!\lfloor D,
\]
which shows that \(\eta_{\delta,g}\) is incompressible. Since Cauchy--Schwarz gives \(\Phi(S^3)\subset C([0,\pi],\overline D)\), the push-forward is a Borel measure on the admissible path space. Observe that the endpoints are fixed by the same spatial marginal. Indeed,
\[
        \Phi(x,v)(0)=x,
        \qquad
        \Phi(x,v)(\pi)=-x.
\]
From \((\pi_x)_\#\mu_{\delta,g}=\Leb^2\!\lfloor D\), we deduce
\[
        (e_0,e_\pi)_\#\eta_{\delta,g}
        =
        (x,-x)_\#(\Leb^2\!\lfloor D).
\]
This implies that \(\eta_{\delta,g}\) joins \(i_D\) to \(-i_D\). We now compute the action. For \(\mu_{\delta,g}\)-a.e. \((x,v)\in S^3\), the path \(t\mapsto \Phi(x,v)(t)=x\cos t+v\sin t\) is absolutely continuous and satisfies
\[
        \dot\Phi(x,v)(t)=-x\sin t+v\cos t .
\]
Since
\[
        \int_0^\pi \sin t\cos t\,dt=0,
        \qquad
        \int_0^\pi \sin^2t\,dt
        =
        \int_0^\pi \cos^2t\,dt
        =
        \frac{\pi}{2},
\]
we have
\[
        \mathcal E(\Phi(x,v))
        =
        \int_0^\pi \frac12|\dot\Phi(x,v)(t)|^2\,dt
        =
        \frac{\pi}{4}(|x|^2+|v|^2)
        =
        \frac{\pi}{4}.
\]
Therefore,
\[
        \mathcal A(\eta_{\delta,g})
        =
        \int_{S^3}\mathcal E(\Phi(x,v))\,d\mu_{\delta,g}(x,v)
        =
        \frac{\pi^2}{4}.
\]

We now turn to optimality. It suffices to consider competitors with finite kinetic
action. Let \(\nu\) be such a competitor. This implies that \(\nu\) is concentrated
on \(H^1([0,\pi],\R^2)\). Since
\[
        (e_0,e_\pi)_\#\nu=(x,-x)_\#(\Leb^2\!\lfloor D)
\]
is supported on the graph \(\{(a,-a):a\in D\}\), it follows that
\[
        \gamma(\pi)=-\gamma(0)
        \qquad\text{for }\nu\text{-a.e. }\gamma .
\]
By the pathwise calibration lemma, we have
\[
        \int_\Omega J(\gamma)\,d\nu(\gamma)\ge0.
\]
Since \(|\gamma(t)|\le1\) on the path space, the potential term is integrable. By Fubini's theorem and incompressibility, we obtain
\[
\begin{aligned}
        \mathcal A(\nu)
        &=
        \int_\Omega J(\gamma)\,d\nu(\gamma)
        +
        \int_\Omega\int_0^\pi
        \frac12|\gamma(t)|^2\,dt\,d\nu(\gamma)              \\
        &=
        \int_\Omega J(\gamma)\,d\nu(\gamma)
        +
        \int_0^\pi\int_D\frac12|y|^2\,dy\,dt                \\
        &\ge
        \int_0^\pi\int_D\frac12|y|^2\,dy\,dt.
\end{aligned}
\]
From
\[
        \int_D |y|^2\,dy=2\pi\int_0^1 r^3\,dr=\frac{\pi}{2},
\]
we conclude that the lower bound is \(\pi^2/4\), and the action of \(\eta_{\delta,g}\) attains this bound. Therefore, \(\eta_{\delta,g}\) is a minimiser.

Finally, we establish the symmetry criterion. By Lemma \ref{lem:rotation-quotient}
and the invariance of \(\sigma\) under \(\overline R_\theta\),
\[
        \Pi\circ\overline R_\theta^{-1}
        =
        \rho_\theta^{-1}\circ\Pi .
\]
This implies
\[
        d\bigl((\overline R_\theta)_\#\mu_{\delta,g}\bigr)
        =
        \pi\bigl(1+\delta\,g\circ\rho_\theta^{-1}\circ\Pi\bigr)\,d\sigma .
\]
Let us fix \(\theta\). Since \(\delta\ne0\), invariance under \(\overline R_\theta\)
is equivalent to
\[
        g\circ\rho_\theta^{-1}\circ\Pi=g\circ\Pi
        \qquad\sigma\text{-a.e. on }S^3.
\]
By the definition of \(\tau=\Pi_\#\sigma\), this is equivalent to
\[
        g\circ\rho_\theta^{-1}=g
        \qquad
        \text{in }L^\infty(\mathbb S^2_{1/2},\tau).
\]
We now replace \(\theta\) by \(-\theta\) to obtain the criterion in question.

It remains to transfer this criterion from phase space to path space. Notice that the
path map is injective due to
\[
        x=\Phi(x,v)(0),
        \qquad
        v=\Phi(x,v)(\pi/2).
\]
Since \(S^3\) is compact and \(C([0,\pi],\overline D)\) is Hausdorff, \(\Phi\)
is a homeomorphism from \(S^3\) onto the compact set \(\Phi(S^3)\). Moreover,
\[
        \mathcal R_\theta\circ\Phi=\Phi\circ\overline R_\theta,
\]
and \(\mathcal R_\theta\Phi(S^3)=\Phi(S^3)\). Therefore, the inverse
\(\Phi^{-1}:\Phi(S^3)\to S^3\) is Borel and
\[
        (\Phi^{-1})_\#\eta_{\delta,g}=\mu_{\delta,g}.
\]
It follows that
\[
        (\mathcal R_\theta)_\#\eta_{\delta,g}=\eta_{\delta,g}
        \quad\Longleftrightarrow\quad
        (\overline R_\theta)_\#\mu_{\delta,g}=\mu_{\delta,g}.
\]
This proves the path-space statement and completes the proof.
\end{proof}

\begin{remark}[Open-disc convention]
If the admissible path space were \(C([0,\pi],D)\) instead of \(C([0,\pi],\overline D)\), we could use the same measure. Indeed, the boundary-hitting oscillators create a \(\mu_{\delta,g}\)-null set. This implies that \(\eta_{\delta,g}\), initially viewed as a measure on \(C([0,\pi],\overline D)\), is concentrated on \(C([0,\pi],D)\). In simpler terms,
we may modify \(\Phi\) on this null set without changing the endpoints,
incompressibility, or action.
\end{remark}

\section{Applications}

We derive two immediate applications of the tilting criterion. We first show that arbitrary bounded functions on the projected Hopf disc give an infinite-dimensional family of
tilts. Finally, we note that one polynomial density already answers the disc question.

\begin{corollary}[An infinite-dimensional family]
\label{cor:infinite-family}
Let \(\psi\) be bounded and Borel on
\[
        \overline B_{1/2}(0):=\{(n,m)\in\R^2:\ n^2+m^2\le1/4\}.
\]
Let
\[
        g_\psi(n,m,\ell):=\ell\,\psi(n,m)
        \qquad\text{on }\mathbb S^2_{1/2}.
\]
It follows that \(g_\psi\) is odd in the \(\ell\)-variable. Hence it satisfies the
oddness hypothesis in Theorem \ref{thm:first-integral-tilts}, and
\[
        d\mu_{\delta,\psi}
        =
        \pi\bigl(1+\delta\,L\,\psi(N,M)\bigr)\,d\sigma
\]
induces a stationary action-minimising generalised incompressible flow from
\(i_D\) to \(-i_D\) whenever
\[
        1+\delta L\psi(N,M)\ge0
        \qquad \sigma\text{-a.e.}
\]
In particular, this condition holds if
\[
        |\delta|\le
        \frac{2}{\|\psi\|_{L^\infty(\overline B_{1/2},\Leb^2)}},
\]
where we interpret the right-hand side as \(+\infty\) if the denominator
vanishes. To describe the symmetry criterion on the projected disc, let us define
\[
        Q_\varphi(n,m)
        :=
        (n\cos\varphi-m\sin\varphi,\,
        n\sin\varphi+m\cos\varphi).
\]
If \(\delta\ne0\), then the corresponding flow is rotationally invariant
exactly when
\[
        \psi=\psi\circ Q_\varphi
        \qquad\text{in }
        L^\infty(\overline B_{1/2}(0),\Leb^2)
        \quad\text{for every }\varphi\in\R.
\]
In other words, \(\psi\) has a radial representative up to \(\Leb^2\)-null sets. In particular, if \(\psi\) is not a.e. radial, then every admissible \(\delta\ne0\) yields a non-rotationally invariant minimiser.
\end{corollary}

\begin{proof}
Oddness in \(\ell\) follows immediately. Let us first verify the positivity range.
By Lemma \ref{lem:hopf-pushforward}, we know that
\(P_\#\tau\), where \(P(n,m,\ell)=(n,m)\), which is equivalent to \(\Leb^2\!\lfloor \overline B_{1/2}(0)\). Therefore,
\[
        |\psi(N,M)|
        \le
        \|\psi\|_{L^\infty(\overline B_{1/2},\,\Leb^2)}
        \qquad \sigma\text{-a.e.}
\]
Since \(|L|\le 1/2\) on \(\mathbb S^2_{1/2}\), we get
\[
        |L\,\psi(N,M)|
        \le
        \frac12\|\psi\|_{L^\infty(\overline B_{1/2},\,\Leb^2)}
        \qquad \sigma\text{-a.e.}
\]
This gives the required sufficient condition for non-negativity.

We will now translate the symmetry criterion to the disc. Again, by
Lemma \ref{lem:hopf-pushforward}, \(\tau:=\Pi_\#\sigma\) is a normalised surface
measure on \(\mathbb S^2_{1/2}\). Moreover, \(P_\#\tau\) is equivalent to
two-dimensional Lebesgue measure \(dA:=d\Leb^2\) on \(\overline B_{1/2}(0)\), and \(\tau(\{\ell=0\})=0\). For \(g_\psi(n,m,\ell)=\ell\psi(n,m)\), the criterion in
Theorem \ref{thm:first-integral-tilts} becomes
\[
        \psi=\psi\circ Q_\varphi
        \qquad\text{in }
        L^\infty(\overline B_{1/2}(0),dA)
        \quad\text{for every }\varphi\in\R.
\]
Indeed, physical rotations act on \((N,M)\) by \(Q_{2\theta}\). The factor
\(\ell\) is nonzero for \(\tau\)-a.e. point, and \(2\theta\) ranges over all
angles.

We must still identify the functions that satisfy these identities. Since the
disc has finite measure, the \(L^\infty\)-identities imply the same identities
in \(L^2\). Let
\[
        \Gamma:=2\pi\Q/2\pi\Z .
\]
For each \(\varphi\in\Gamma\), we choose a null set outside which
\(\psi=\psi\circ Q_\varphi\) holds, and remove the countable union of these
sets. By Fubini's theorem in polar coordinates, we obtain a set \(E\subset(0,1/2)\) of full one-dimensional measure such that, for each \(r\in E\), the section
\[
        \psi_r(\vartheta)
        :=
        \psi(r\cos\vartheta,r\sin\vartheta)
\]
belongs to \(L^2(\R/2\pi\Z)\) and is invariant, in \(L^2\), under translation
by every element of \(\Gamma\). Recall that translations are strongly continuous on
\(L^2(\R/2\pi\Z)\), so the section \(\psi_r\) is invariant under all rotations. It follows that all non-zero Fourier coefficients of \(\psi_r\) vanish, and
\(\psi_r\) is a.e. constant. Let us define
\[
        h(r):=\frac{1}{2\pi}\int_0^{2\pi}
        \psi(r\cos\vartheta,r\sin\vartheta)\,d\vartheta
\]
whenever the section belongs to \(L^2\), and define \(h\) arbitrarily on the
remaining radii. This implies that \(h(|\cdot|)\) is a measurable radial representative of
\(\psi\). The converse follows immediately, since every radial representative satisfies the required rotation identities.
\end{proof}

\begin{corollary}[A fixed-dimensional answer to the disc question]
\label{cor:polynomial-example}
For every \(0<|\eps|<8\), the phase measure
\[
        d\mu_\eps
        =
        \pi\bigl(1+\eps LM\bigr)\,d\sigma
\]
induces a minimising generalised incompressible flow from \(i_D\) to \(-i_D\)
in the unit disc. Furthermore, this minimiser is not rotationally invariant. In coordinates,
\[
        LM=(x_1v_2-x_2v_1)(x_1x_2+v_1v_2).
\]
\end{corollary}

\begin{proof}
We apply Corollary \ref{cor:infinite-family} to \(\psi(n,m)=m\). The
non-negativity condition follows from a direct estimate. Since
\[
        L^2+M^2+N^2=\frac14,
\]
we have
\[
        |LM|\le\frac{L^2+M^2}{2}\le\frac18.
\]
It follows that \(1+\eps LM>0\) whenever \(|\eps|<8\). It remains to exclude
axisymmetry. Since
\[
        \rho_{\pi/2}(n,m,\ell)=(-n,-m,\ell),
\]
we have \(g\circ\rho_{\pi/2}=-g\) for \(g(n,m,\ell)=\ell m\). The function
\(g\) is continuous and nonzero on the open set \(\{\ell m\ne0\}\subset\mathbb S^2_{1/2}\). Therefore, \(g\ne0\) holds in \(L^\infty(\tau)\). If \(g\) were axisymmetric, then \(g\circ\rho_{\pi/2}=-g\) would force \(g=0\) \(\tau\)-a.e., which would lead to a contradiction. We conclude that the tilted measure is not rotationally invariant whenever
\(\eps\ne0\).
\end{proof}

We have thus shown that the question posed in \cite[p. 144, before Section 4.1]{BFS} and reiterated in \cite[p. 37, end of Section 1.4.4]{DaneriFigalli2013} has an affirmative answer already in the stationary, strictly positive, and single-energy-shell
class.

\subsection{Comparison with fixed orientation}
\label{sec:BFS-comparison}
For completeness, we compare the construction with the orientation-fixed framework of
\cite{BFS}, and fix the sign convention between their orientation variable and
the Hopf variable \(L\). We keep the mass-\(\pi\) convention from the introduction. Under this convention, the incompressible marginal is \(\Leb^2\!\lfloor D\), and the
isotropic single-shell law is \(\mu_0=\pi\sigma\). In the normalised-area convention, we divide all phase-space and path-space measures below by \(\pi\). As in the proof of Theorem \ref{thm:first-integral-tilts}, we replace \(g\), when necessary, by its pointwise \(\kappa\)-odd representative, where \(\kappa(n,m,\ell)=(n,m,-\ell)\). This replacement keeps \(\mu_{\delta,g}\) the same. We use this representative in the
section below.

In the notation introduced just before \cite[Proposition 4.5]{BFS}, the
projection in the disc case \(R=0\) is
\[
        \pi_P(x,v)=\left(|x|^2-\frac12,\,x\cdot v\right).
\]
Note that the map \(\pi_P\) is invariant under physical rotations, but rotates under the
oscillator flow. On the other hand, \(\Pi=(N,M,L)\) is invariant under the oscillator
flow and transforms under physical rotations. It follows that the Hopf quotient is the
natural coordinate system for this construction. Indeed, our densities are first
integrals of the oscillator flow. They are also constant along oscillator orbits,
but need not be constant along physical rotation orbits. In this sense, our
quotient is transverse to the quotient used in \cite[Section 4.2]{BFS}. The map \(\pi_P\) is adapted to rotationally invariant disintegrations, while \(\Pi\) is adapted to stationary disintegrations.

On \(S^3\), this projection is
\[
        A:=|x|^2-\frac12=\frac12(|x|^2-|v|^2),
        \qquad B:=x\cdot v.
\]
The oscillator flow rotates \((A,B)\) with angular speed \(2\). The orientation variable of \cite[Section 4.1]{BFS} uses
\[
        x^\perp=(x_2,-x_1).
\]
With our rotation convention, \(x^\perp=R_{-\pi/2}x\), so
\[
        x^\perp\cdot v=x_2v_1-x_1v_2=-L .
\]
In this convention, the clockwise/anticlockwise decomposition of
\cite[Section 4.1]{BFS} is precisely the decomposition by the sign of \(L\), where the labels are interchanged. With the convention \(\operatorname{sgn}0=0\), the choices
\(g=\operatorname{sgn}\ell\), \(\delta=1\), and \(g=-\operatorname{sgn}\ell\), \(\delta=1\), give
\[
        2\,\mu_0\!\lfloor_{\{L>0\}}
        =
        \pi(1+\operatorname{sgn}L)\sigma,
        \qquad
        2\,\mu_0\!\lfloor_{\{L<0\}}
        =
        \pi(1-\operatorname{sgn}L)\sigma .
\]
After division by \(\pi\), these are the corresponding probability laws. It
follows that the orientation classes in \cite{BFS} and \cite{DaneriFigalli2013} satisfy
\[
        TD^+=\{L<0\},
        \qquad
        TD^-=\{L>0\}.
\]
Therefore, up to this sign convention, \(2\sigma\!\lfloor_{\{L<0\}}\) and
\(2\sigma\!\lfloor_{\{L>0\}}\) are, respectively, the normalised oriented laws
\(\mu^+\) and \(\mu^-\) in \cite{DaneriFigalli2013}.

We now analyse the corresponding \(\pi_P\)-push-forward in our normalisation, for comparison with \cite[Proposition 4.5]{BFS}. The computation is direct and does not rely on the constants from \cite{BFS}. Let us define
\[
        Q:S^3\to\mathbb S^2_{1/2},
        \qquad
        Q(x,v):=(A(x,v),B(x,v),L(x,v)).
\]
Let us set \(w_1=x_1+ix_2\) and \(w_2=v_1+iv_2\). It follows that
\[
        A=\frac12(|w_1|^2-|w_2|^2),
        \qquad
        w_1\overline{w_2}=B-iL .
\]
We infer that \(Q\) is the Hopf quotient for the physical rotation action.
Indeed, equality of the \(Q\)-coordinates is equivalent, in these complex
variables, to equality of the rank-one Hermitian matrices \(ww^*\). This implies that two
points in the same \(Q\)-fibre differ by a common phase, namely, by a physical
rotation. By the proof of Lemma \ref{lem:hopf-pushforward}, \(Q_\#\sigma\) is
a normalised surface measure on \(\mathbb S^2_{1/2}\). The physical rotation
action also preserves \(\sigma\) and is transitive on the fibres of \(Q\). Therefore, by the uniqueness of Haar measure on compact orbits, the Rokhlin
disintegration of \(\sigma\) with respect to \(Q\) is the normalised Haar
measure on the corresponding physical-rotation orbit, for \(Q_\#\sigma\)-a.e. base point.

We now disintegrate this normalised surface measure over
\(P_Q(A,B,L)=(A,B)\). Observe that the two sheets
\[
        L=\pm\sqrt{1/4-A^2-B^2}
\]
have equal conditional weights for \(A^2+B^2<1/4\). On either sheet, Haar
measure on the physical-rotation orbit is the conditional measure of
\(\sigma\). As the physical rotation angle varies, Lemma \ref{lem:rotation-quotient}
shows that \((N,M)\) moves uniformly on the circle \(N^2+M^2=A^2+B^2\), while \(L\) remains fixed. More explicitly, for \(A^2+B^2<1/4\), we set
\[
        r(A,B):=\sqrt{A^2+B^2},
        \qquad
        \ell(A,B):=\sqrt{1/4-A^2-B^2}.
\]
The conditional average of \(g\circ\Pi\) over the sheet \(L=\ell(A,B)\) in the
\(Q\)-fibre is
\[
        \frac1{2\pi}\int_0^{2\pi}
        g\bigl(r(A,B)\cos\vartheta,\,
               r(A,B)\sin\vartheta,\,
               \ell(A,B)\bigr)\,d\vartheta .
\]
By the \(\kappa\)-oddness of \(g\), the conditional average over the sheet
\(L=-\ell(A,B)\) is the negative of this number. Since the two sheets have equal conditional weights, the two contributions cancel each other out. Therefore, for every
bounded Borel \(F\) on the \((A,B)\)-disc, we have
\[
        \int_{S^3}F(A,B)\,g(\Pi)\,d\sigma=0.
\]
It follows that
\[
        (\pi_P)_\#\bigl((g\circ\Pi)\sigma\bigr)=0,
\]
and hence
\[
        (\pi_P)_\#\mu_{\delta,g}=(\pi_P)_\#\mu_0 .
\]
Our construction uses exactly the conditional freedom that remains after
disintegration over the fibres of \(\pi_P\). This matches the context of \cite[Corollary 4.6 and Theorem 4.20]{BFS}: those uniqueness statements fix
the clockwise orientation class, and the anticlockwise statement follows by
reversing orientation. The polynomial example above, and, more generally, every
strictly positive tilt from Theorem \ref{thm:first-integral-tilts}, charges both \(\{L>0\}\) and \(\{L<0\}\). The odd cancellation under \(v\mapsto -v\) implies incompressibility, which shows that this result is not bound to a single orientation class.
\bibliographystyle{unsrt}
\bibliography{br}

\end{document}